\documentclass[11pt]{article}
\usepackage[numbers,sort&compress]{natbib}
\usepackage{enumerate}
\usepackage{amscd}
\usepackage{amsmath}
\usepackage{latexsym}
\usepackage{amsfonts}
\usepackage{amssymb}
\usepackage{amsthm}
\usepackage{verbatim}
\usepackage{mathrsfs}
\usepackage{enumerate}
\usepackage{hyperref}

\theoremstyle{plain}\newtheorem{definition}{Definition}[section]
\theoremstyle{definition}\newtheorem{theorem}{Theorem}[section]
\theoremstyle{plain}\newtheorem{lemma}[theorem]{Lemma}
\theoremstyle{plain}\newtheorem{coro}[theorem]{Corollary}
\theoremstyle{plain}
\theoremstyle{remark}\newtheorem{remark}{Remark}[section]
\usepackage{xcolor}

\newcommand{\Div}{\mathrm{div}\,}
\newcommand{\B}{\Big}

\newcommand{\be}{\begin{equation}}
\newcommand{\ee}{\end{equation}}
 \newcommand{\ba}{\begin{aligned}}
 \newcommand{\ea}{\end{aligned}}

  \newcommand{\f}{\frac}
    
  \newcommand{\ben}{\begin{enumerate}}
   \newcommand{\een}{\end{enumerate}}

\newcommand{\Rmnum}[1]{\expandafter\@slowromancap\romannumeral #1@}

\allowdisplaybreaks

\numberwithin{equation}{section}
\begin{document}
%%%%%%%%%%%%%%%%%%%%%%%%%%%%%%%%%%%%%%%%%%%%%%%%%%%%%%%%%%%%%%%%%%%%%%%%%%%%%%%%%%%%%%%%%%%%%%%%%%%%
\title{Energy equality for       the isentropic compressible
Navier-Stokes equations without upper bound of the density }
\author{  Yulin Ye\footnote{ School of Mathematics and Statistics,
Henan University,
Kaifeng, 475004,
P. R. China. Email: ylye@vip.henu.edu.cn}, ~ Yanqing Wang\footnote{Corresponding author.   College of Mathematics and   Information Science, Zhengzhou University of Light Industry, Zhengzhou, Henan  450002,  P. R. China Email: wangyanqing20056@gmail.com}  ~   \, \, and \, Huan Yu \footnote{ School of Applied Science, Beijing Information Science and Technology University, Beijing, 100192, P. R. China Email:  yuhuandreamer@163.com}
 }
\date{}
\maketitle
\begin{abstract}
In this paper, we  are concerned with the minimal regularity of both the density
and the velocity for the weak solutions keeping energy equality
in the   isentropic compressible Navier-Stokes equations. The energy equality
 criteria
 without upper bound of the density    are established.
  Almost all previous corresponding  results  requires $\varrho\in L^{\infty}(0,T;L^{\infty}(\mathbb{T}^{d}))$.
  \end{abstract}
\noindent {\bf MSC(2000):}\quad 35Q30, 35Q35, 76D03, 76D05\\\noindent
{\bf Keywords:} compressible Navier-Stokes equations;  energy equality;  vacuum %%%%%%%%%%
\section{Introduction}
\label{intro}
\setcounter{section}{1}\setcounter{equation}{0}

The  classical isentropic compressible Navier-Stokes equations
\be\left\{\ba\label{INS}
&\varrho_t+\nabla \cdot (\varrho v)=0, \\
&(\varrho v)_{t} +\Div(\varrho v\otimes v)+\nabla
P(\varrho )-  \text{div\,}(\mu\mathfrak{D}v)- \nabla(\lambda\text{div\,}v)=0,\\
\ea\right.\ee
where $\varrho$ stands for the density of the flow, $v$
reprents the flow  velocity field and $P(\varrho)=\varrho^{\gamma}$ is the scalar pressure;
The viscosity coefficients $\mu$ and $\lambda$ satisfy $\mu\geq 0$ and $2\mu+d\lambda>0$.
$\mathfrak{D}v=\f12(\nabla v\otimes \nabla v^{T} )$ is the stain tensor;
We complement equations \eqref{INS} with initial data
\begin{equation}\label{INS1}
	\varrho(0,x)=\varrho_0(x),\ (\varrho v)(0,x)=(\varrho_0 v_0)(x),\ x\in \Omega,
\end{equation}
where we define $v_0 =0$ on the sets $\{x\in \Omega:\ \varrho_0=0\}.$ In the present paper, we consider the periodic case, that is $\Omega=\mathbb{T}^d$ with dimension $d\geq 2$.

One of the celebrated results of the  isentropic compressible Navier-Stokes equations \eqref{INS}  is the global existence of
the   finite energy   weak solutions due to   Lions  \cite{[Lions1]} with $\gamma\geq \f{3d}{d+2}$ for $d=2$ or $3$, where $d$ is the spatial dimension.   Subsequently, in \cite{FNP}, Feireisl-Novotny-Petzeltov\'a further extended the Lions' work  to $\gamma >\f{3}{2}$ for $d=3$. In \cite{JZ}, Jiang and Zhang considered the global existence of weak solutions for the case $\gamma>1$
with the spherical symmetric initial data. For the convenience of the reader, we recall the definition of the   finite energy   weak solutions:
\begin{definition}\label{wsdefi}
 	A pair ($\varrho,v$) is called a weak solution to \eqref{INS} with initial data \eqref{INS1} if ($\varrho,v$) satisfies
 	\begin{enumerate}[(i)]
 		\item equation \eqref{INS} holds in $D'(0,T;\Omega)$ and
 		\begin{equation}P(\varrho ), \varrho |v|^2\in L^\infty(0,T;L^1(\Omega)),\ \ \ \nabla v\in L^2(0,T;L^2(\Omega)),
 		\end{equation}
 	\item[(ii)]
 	the density $\varrho$ is a renormalized solution of \eqref{INS} in the sense of \cite{[PL]}.
 		\item[(iii)]
 		the energy inequality holds
 		\begin{equation}\label{energyineq}
 		\begin{aligned}
 	E(t) +\int_0^T\int_{\Omega}\left( \mu|\nabla v|^2+(\mu+\lambda)|\Div v|^2 \right) dxdt \leq E(0), 		\end{aligned}\end{equation}
 where $E(t)=\int_{\Omega}\left( \frac{1}{2}\varrho |v|^2+\f{\varrho^{\gamma}}{\gamma-1} \right) dx$.
 	\end{enumerate}
 \end{definition}
Note that the   finite energy   weak solutions  constructed in \cite{[Lions1],FNP}    satisfy
the energy inequality \eqref{energyineq}. A natural question is
how much regularity of weak solutions in
these system are required to ensure   energy  equality.
In \cite{[Yu2]},
Yu opened the door of research of  the  Lions-Shinbrot type  energy equality criterion
 to the weak solutions of the  compressible Navier-Stokes  equations   with both non vacuum and vacuum cases.
 The so-called    Lions-Shinbrot type criterion is that if
 the   Leray-Hopf weak solutions $v$ of the 3D incompressible  Navier-Stokes equations satisfy
$
v\in L^{p}(0,T;L^{q}(\mathbb{T}^d)),~\text{with}~\f{2}{p}+
 \f{2}{q}=1~\text{and}~q\geq 4,$
then the following corresponding energy  equality holds
 $$
 \|v\|_{L^{2}(\Omega)}^{2}+2 \int_{0}^{T}\|\nabla v\|_{L^{2}(\Omega)}^{2}ds= \|v_0\|_{L^{2}(\Omega)}^{2}.
 $$
 Yu's results in \cite{[Yu2]} without  vacuum can be formulated as: if a weak solution
 $(\varrho, v)$ of  the compressible Navier-Stokes equations \eqref{INS}  with $\mu =2\varrho$ and $\lambda =0$
  satisfies
\be\ba\label{yu1}
\sqrt{\varrho}v\in L^{\infty}(0,T;L^{2}(\Omega)),\sqrt{\varrho}\nabla v\in L^{2}(0,T;L^{2}(\Omega)),\\
0<c_{1}\leq\varrho\leq c_{2}<\infty, \ \ \nabla\sqrt{\varrho} \in L^{\infty}(0,T;L^{2}(\Omega)),\\
v\in L^{p}(0,T;L^{q}(\Omega)) ~~\f1p+\f1q\leq \f5{12}   \ \text{ and}\   q\geq6, \sqrt{\varrho_{0}} v_{0} \in L^{4}(\Omega),
\ea
\ee
 then such a weak solution $(\varrho, v)$ for any $t \in[0, T]$ satisfies
 \begin{equation}\label{EIy}\ba
\int_{\mathbb{T}^{d}}\left( \frac{1}{2}\varrho |v|^2+\f{\varrho^{\gamma}}{\gamma-1} \right) dx +\int_0^t\int_{\mathbb{T}^{d}}& \varrho|\mathfrak{D}v|^2   dxdt = \int_{\mathbb{T}^{d}}\left( \frac{1}{2} \varrho_0 |v_0|^2+\f{\varrho_{0}^{\gamma}}{\gamma-1}\right)dx.
\ea
\end{equation}
 Subsequently,  Nguye-Nguye-Tang  \cite{[NNT]}  obtained  the following
  Lions-Shinbrot type criterion for the weak solutions
  of the compressible Navier-Stokes equations \eqref{INS}  with $\mu=\mu(\varrho)$, $\lambda=\lambda(\varrho)$ and general pressure law $P(\varrho)\in C^{2}(0,+\infty)$,
 \be\ba\label{NNT}
 &0<  c_{1}\leq\varrho\leq c_{2}<\infty,  v\in L^{\infty}(0,T; L^{2}
 (\mathbb{T}^{d})), \nabla v\in L^{2}(0,T; L^{2}(\mathbb{T}^{d})),   \\&\sup_{t\in(0,T)}\sup_{|h|<\varepsilon}|h|^{-\f12}
 \|\varrho(\cdot+h,t)-\varrho(\cdot,t)\|_{L^{2}(\mathbb{T}^{d})}<\infty,\\
 &v\in L^{4}(0,T;L^{4}(\mathbb{T}^{d})).
 \ea\ee
 ensures that the corresponding energy equality  is valid.  Very recently, in \cite{[WY]}, the authors
refined  \eqref{NNT}     to
 \be \ba\label{wy}
&0<  c_{1}\leq\varrho\leq c_{2}<\infty,   v\in L^{\infty}(0,T; L^{2}(\mathbb{T}^{d})),   \nabla v\in L^{2}(0,T; L^{2}(\mathbb{T}^{d})),   \\
& v\in L^{\f{2p}{p-1}}(0,T; L^{\f{2q}{q-1}}(\mathbb{T}^{d})),
 \nabla v\in L^{p}(0,T; L^{q}(\mathbb{T}^{d})).
\ea \ee
It is worth pointing out that
a special case $p=q=2$ in \eqref{wy}
 is an improvement of \eqref{NNT}. Other kind Lions-Shinbrot type criterion
   for the  compressible
 Navier-Stokes system \eqref{INS}  can be found in \cite{[Liang],[WY]}. To the knowledge of
 the authors, all  Lions-Shinbrot type criterion for the energy equality in the compressible Navier-Stokes system
   require  the upper   bound of the density.
   One goal of this paper is to  analysis  the role of   the integrability of    density
    in the energy  equality for the weak solutions of the   compressible Navier-Stokes equations.
In the case of no vacuum,  we formulate our first result  as follows:
 \begin{theorem}\label{the1.2}
 Let $(\varrho, v)$ be a weak solutions   in the sense of definition \eqref{wsdefi}. Assume
 \begin{equation} \label{u}\begin{aligned}
&0<  c\leq\varrho \in L^{ k}(0,T;L^{ l}(\mathbb{T}^d)), \\
& v\in L^{p }(0,T;L^{q} (\mathbb{T}^{d})),\nabla v\in L^{\f{kp}{k(p-2)-4p}}(0,T; L^{\f{lq}{l(q-2)-4q}}(\mathbb{T}^{d})),
\end{aligned}
\end{equation}
with $k> \max\{\frac{4p}{p-2},\f{(\gamma-2)p }{2}\}$, $l> \max\{\frac{4q}{q-2}, \f{(\gamma-2)q}{2}\}$ and $k(4-p)+7p\geq 0,\  l(4-q)+7q\geq 0$.
the energy equality below  is valid, for any $t\in[0,T]$,
\begin{equation}\label{EI}\ba
\int_{\mathbb{T}^{d}}\left( \frac{1}{2}\varrho |v|^2+\f{\varrho^{\gamma}}{\gamma-1} \right) dx +\int_0^t\int_{\mathbb{T}^{d}}&\left[\mu|\nabla v|^2+(\mu+\lambda)|\Div v|^2 \right] dxdt\\=&\int_{\mathbb{T}^{d}}\left( \frac{1}{2} \varrho_0 |v_0|^2+\f{\varrho_{0}^{\gamma}}{\gamma-1}\right)dx.\ea\end{equation}
\end{theorem}
\begin{remark}
This theorem is a generalization of recent wroks \cite{[WY],[NNT]}. It seems that  this is the first energy equality criterion for the compressible
Navier-Stokes equations without the upper bound of the density.
\end{remark}
\begin{remark}
A special case of \eqref{u} is that     $v\in L^4(0,T;L^4(\mathbb{T}^d)),$  $\nabla v\in L^{\frac{2k}{k-8}}(0,T;L^{\frac{2l}{l-8}})$ and $\varrho \in L^k(0,T;L^l(\mathbb{T}^d))$ with $k>\max\{8, 4(\gamma-2)\} $ and $l>\max\{8, 4(\gamma-2)\}$ guarantee the energy conservation of the weak solutions.  This shows that  lower integrability of the density  $\varrho$ means that more integrability of gradient  of the velocity  $\nabla v$ is necessary in energy conservation of the  isentropic compressible fluid with non-vacuum and the inverse is also true.
\end{remark}
Next, we turn our attentions to  the energy equality of the  isentropic Navier-Stokes equations allowing
  vacuum.
In  aforementioned work \cite{[Yu2]},
Yu  also deat with the energy conservation of system \eqref{INS} in the present  of vacuum and proved that  if a weak solution $(\varrho, v)$  in the sense of Definition  \ref{wsdefi} satisfies
\be\ba\label{yu}
&0\leq\varrho\leq c<\infty, \ \ \nabla\sqrt{\varrho} \in L^{\infty}(0,T;L^{2}(\mathbb{T}^{d})),\\
&v\in L^{p}(0,T;L^{q}(\mathbb{T}^{d})) ~~p\geq4   \ \text{ and}\   q\geq6,\ and\ v_0\in L^{q_0},\ q_0\geq 3,
\ea
\ee
then   the   energy equality \eqref{EI}    is valid for $t\in[0,T]$.

Developing the technique used in   \cite{[Yu2]}, the authors
 showed Lions's    $L^4 (0,T; L^4(\Omega))$   condition  for energy balance is also valid for the   weak solutions of  the  isentropic compressible Navier-Stokes equations  allowing vacuum via the following energy equality    criterion: if a weak solution $(\varrho, v)$ satisfies,
for any $p,q\geq 4$ and $dp< 2q+3d$ with $d\geq 2$,
 \be\label{key0}
\ba
&0\leq \varrho<c<\infty, \nabla\sqrt{\varrho}\in L^{\f{2p}{3-p}}(0,T;L^{\f{2q}{3-q}}(\mathbb{T}^{d})),
\\& v\in L^{\f{2p}{p-1}}(0,T;L^{\f{2q}{q-1} }(\mathbb{T}^{d})),
\nabla v\in L^{p}(0,T;L^{q}(\mathbb{T}^{d})),\ v_0\in L^{\f{q}{q-1}}(\mathbb{T}^{d}).
\ea\ee
then there holds \eqref{EI}.
Now, we state our second result for the compressible Navier-Stokes equations \eqref{INS} in the presence of vacuum as follows:
\begin{theorem}\label{the1.5}
 For any dimension $d\geq2$, the energy equality \eqref{EI} of   weak solutions  $(\varrho, v)$ to the compressible Navier-Stokes equation \eqref{INS} with vacuum is valid   for $t\in[0,T]$ provided
\be
  \ba
   &0\leq \varrho\in L^{k}(0,T;L^{l}(\mathbb{T}^d))
,\nabla\sqrt{\varrho}\in L^{\f{2pk}{2k(p-3)-p}}(0,T;L^{\f{2lq}{2l(q-3)-q}}(\mathbb{T}^d))\\
& v\in L^{p }(0,T;L^{q}(\mathbb{T}^d)),\nabla v \in L^{\f{pk}{k(p-2)-p}}(0,T;L^{\f{ql}{l(q-2)-q}}(\mathbb{T}^d))\ and\ v_0\in L^{\max\{
	\frac{2\gamma }{\gamma -1},\frac{q}{2}\}}(\mathbb{T}^d),
  \ea\ee
  where $k> \max\{\frac{p}{2(p-3)},\frac{p}{p-2}, \frac{(\gamma-1)p}{2},\frac{(\gamma-1)(d+q)p}{2q-d(p-3)}\}$, $l> \max\{\frac{q}{2(q-3)},\frac{q}{q-2}, \frac{(\gamma-1)q}{2}\}$, $p>3$ and $q>\max\{3, \frac{d(p-3)}{2}\}$.
  \end{theorem}
\begin{remark}
If let $k=l\rightarrow+\infty$, then Theorem \ref{the1.5} will reduce to the result obtained in \cite{[YWW]}, hence, this result can be seen as a generalization of  recent works in \cite{[Yu2],[YWW]}. To the best of our knowledge, this seems to be the first result of the energy conservation criteria for the weak solutions of compressible Navier-Stokes equations without upper bound of density in the presence of vacuum.
\end{remark}
  \begin{coro}
The energy equality \eqref{EI} of   weak solutions  $(\varrho, v)$ to the compressible Navier-Stokes equation \eqref{INS} allowing vacuum holds  for $t\in[0,T]$  provided
   \begin{enumerate}[(1)]
 \item $\nabla v \in L^{2 }(0,T;L^{2}(\mathbb{T}^{d})),$
$$
  \ba
  &0\leq\varrho\in L^{k}(0,T;L^{l} (\mathbb{T}^{d}))
,\nabla\sqrt{\varrho}\in L^{\f{4k}{ k +4}}(0,T;L^{\f{4l}{  l+4}} (\mathbb{T}^{d})),\\
&   v\in L^{\f{4k}{k - 2}}(0,T;L^{\f{4l}{ l- 2}} (\mathbb{T}^{d})), v_0\in L^{\max\{\frac{2\gamma}{\gamma-1},\frac{2l}{l-2}\}}(\mathbb{T}^d),\\
&with\ k,l> 2\gamma\ and \ k>\frac{2\B(2(d+4)\gamma+d\B)l-4d(2\gamma+1)}{(8-d)l+2d};
  \ea$$
 \item
 $v\in L^4(0,T;L^4(\mathbb{T}^d))$,
 $$
 \ba
 &0\leq \varrho \in L^{k}(0,T;L^{l}(\mathbb{T}^d)), \nabla \sqrt{\varrho }\in L^{\frac{4k}{k-2}}(0,T;L^{\frac{4l}{l-2}}(\mathbb{T}^d)) ,\\
 &  \nabla v\in L^{\frac{2k}{k-2}}(0,T;L^{\frac{2l}{l-2}}(\mathbb{T}^d))\ and\ v_0\in L^{\frac{2\gamma}{\gamma-1}}(\mathbb{T}^d),\\
 & with\ k>\max\{2,\frac{4(\gamma -1)(d+4)}{8-d}\}\ and\ l>\max\{2,2(\gamma-1) \};
 \ea$$
    \end{enumerate}
  \end{coro}
  \begin{remark}
This shows that  lower integrability of the density  $\varrho$ means that more integrability of  the velocity $v$ or the gradient  of the velocity  $\nabla v$ are necessary in energy conservation of the  isentropic compressible fluid with vacuum and the inverse is also true.
  \end{remark}

The minimum regularity of weak solution keeping
energy conservation in the fluid equations  originated from Onsager's work \cite{[Onsager]}.
The recent progrss of Onsager's conjecture can be found in  \cite{[CY],[CET],[NNT1],[ADSW]}
We refer the readers to \cite{[WYY]} for  the   Onsager conjecture on the energy conservation for
 the isentropic compressible Euler equations via establishing
  the    energy conservation   criterion  involving the density $\varrho\in L^{k}(0,T;L^{l}(\mathbb{T}^{d}))$.
There are very significant
recent developments on  energy  equality of the Leray-Hopf weak solutions   of the 3D incompressible
Navier-Stokes equations in \cite{[WWY],[BY],[BC],[CL],[Zhang],[Yu1],[Yu3]}.

The paper is organized as follows: We present some auxiliary lemmas for dealing with the density without upper bound of the density. In section 3, we prove sufficient conditions for energy equality of   weak solutions keeping energy equality in the isentropic compressible Navier-Stokes equations in the absence of vacuum. Finally,   section 4 is  devoted to the  energy equality criterion in the presence of vacuum.
\section{Notations and some auxiliary lemmas} \label{section2}

First, we introduce some notations used in this paper.
For $p\in [1,\,\infty]$, the notation $L^{p}(0,\,T;X)$ stands for the set of measurable functions on the interval $(0,\,T)$ with values in $X$ and $\|f(t,\cdot)\|_{X}$ belonging to $L^{p}(0,\,T)$. The classical Sobolev space $W^{k,p}(\mathbb{T}^d)$ is equipped with the norm $\|f\|_{W^{k,p}(\mathbb{T}^d)}=\sum\limits_{|\alpha| =0}^{k}\|D^{\alpha}f\|_{L^{p}(\mathbb{T}^d)}$.   For simplicity, we denote by $$\int_0^T\int_{\mathbb{T}^d} f(t, x)dxdt=\int_0^T\int f\ ~~\text{and}~~ \|f\|_{L^p(0,T;X )}=\|f\|_{L^p(X)}.$$

On the one hand, if we let  $\eta$ be non-negative   smooth function only supported in the space ball of radius 1 and its integral equals to 1, we define the rescaled space mollifier $\eta_{\varepsilon}(x)=\frac{1}{\varepsilon^{d}}\eta(\frac{x}{\varepsilon})$ and
$$
f^{\varepsilon}(x)=\int_{\mathbb{T}^d}f(y)\eta_{\varepsilon}(x-y)dy,
$$
then we have the following  standard   properties of the mollifier kernel:
 \begin{lemma}\label{lem2.11}
	Let $ p,q,p_1,q_1,p_2,q_2\in[1,+\infty)$ with $\frac{1}{p}=\frac{1}{p_1}+\frac{1}{p_2},\frac{1}{q}=\frac{1}{q_1}+\frac{1}{q_2} $. Assume $f\in L^{p_1}(0,T;L^{q_1}(\mathbb{T}^d)) $ and $g\in L^{p_2}(0,T;L^{q_2}(\mathbb{T}^d))$. Then for any $\varepsilon>0$, there holds
	\begin{equation}\label{a4}
	\|(fg)^\varepsilon-f^\varepsilon g^\varepsilon\|_{L^p(0,T;L^q(\mathbb{T}^d))}\rightarrow 0,\ \ \ as\ \varepsilon\rightarrow 0.
	\end{equation}
\end{lemma}
 We also recall the general    Constantin-E-Titi type
and Lions type commutators on mollifying kernel proved in \cite{[WY],[NNT]}.
\begin{lemma} \label{lem2.2}   Let $1\leq p,q,p_1,p_2,q_1,q_2\leq \infty$  with $\frac{1}{p}=\frac{1}{p_1}+\frac{1}{p_2}$ and $\frac{1}{q}=\frac{1}{q_1}+\frac{1}{q_2}$. Assume $f\in L^{p_1}(0,T;W^{1,q_1}(\mathbb{T}^d))$ and $g\in L^{p_2}(0,T;L^{q_2}(\mathbb{T}^d))$. Then for any $\varepsilon> 0$, there holds
		\begin{align} \label{fg'}
		\|(fg)^\varepsilon-f^\varepsilon g^\varepsilon\|_{L^p(0,T;L^q(\mathbb{T}^d))}\leq C\varepsilon \|\nabla f\|_{L^{p_1}(0,T;L^{ q_1}( \mathbb{T}^d))}\|g\|_{L^{p_2}(0,T;L^{q_2}(\mathbb{T}^d))}.
		\end{align}
		Moreover, if $q_1,q_2<\infty$ then
		\begin{align}\label{limite'}
		\limsup_{\varepsilon \to 0}\varepsilon^{-1} \|(fg)^\varepsilon-f^\varepsilon g^\varepsilon\|_{L^p(0,T;L^q(\mathbb{T}^d))}=0.
		\end{align}
	\end{lemma}
The following two lemmas plays an    important role in the proof of Theorem \ref{the1.2}.
We refer the reader to \cite{[NNT],[Liang]}
  for their original version.
\begin{lemma} \label{lem2.1}
Suppose that $f\in L^{p}(0,T;L^{q}(\mathbb{T}^{d}))$. Then for any $\varepsilon>0$, there holds
\be\label{2.1}
\|\nabla f^{\varepsilon}\|_{L^{p}(0,T;L^{q}(\mathbb{T}^{d}))}
\leq C\varepsilon^{-1}\| f\|_{L^{p}(0,T;L^{q}(\mathbb{T}^{d}))},
\ee
and, if $p,q<\infty$
\begin{equation}\label{2.2}
\limsup_{\varepsilon\rightarrow0} \varepsilon\|\nabla f^{\varepsilon}\|_{L^{p}(0,T;L^{q}(\mathbb{T}^{d}))}=0.
\end{equation}
Moreover, if $0<c_{1}\leq g\in L^k(0,T;L^{l}(\mathbb{T}^d))$, then there holds, for any $\varepsilon>0$,
\begin{equation}\label{2.3}\begin{aligned}
\B\|\nabla \f{f^{\varepsilon}}{g^{\varepsilon}}\B\|_{L^{\frac{pk}{k+p}}(0,T;L^{\frac{ql}{l+q}}(\mathbb{T}^{d}))}\leq C
\varepsilon^{-1}\|f\|_{L^{p}(0,T;L^{q}(\mathbb{T}^{d}))}\|g\|_{L^k(0,T;L^{l}(\mathbb{T}^d))},
\end{aligned}\end{equation}
and if $1\leq q,l<\infty$
\begin{equation}\label{2.4}
\limsup_{\varepsilon\rightarrow0} \varepsilon\B\|\nabla \f{f^{\varepsilon}}{g^{\varepsilon}}\B\|_{L^{\frac{pk}{k+p}}(0,T;L^{\frac{ql}{l+q}}(\mathbb{T}^{d}))}
=0.\end{equation}
\end{lemma}
\begin{proof}
	The proof of \eqref{2.1} and \eqref{2.2} was given in \cite{[NNT]}, hence we only focus on the proof of \eqref{2.3}-\eqref{2.4} here.
	
	By direct calculation and using the lower bound of $g$, we have
	\begin{equation}\label{2.5}
		\begin{aligned}
			\B\|\nabla (\frac{f^\varepsilon}{g^\varepsilon})\B\|_{L^{\frac{ql}{l+q}}}&\leq C\left(\B\|\frac{g^\varepsilon}{(g^\varepsilon)^2}\nabla f^\varepsilon\B\|_{L^{\frac{ql}{l+q}}}+\B\|f^\varepsilon\frac{\nabla g^\varepsilon}{(g^\varepsilon)^2}\B\|_{L^{\frac{ql}{l+q}}}\right)\\
			&\leq C\left(\|g^\varepsilon \nabla f^\varepsilon\|_{L^{\frac{ql}{l+q}}}+\|f^\varepsilon \nabla g^\varepsilon\|_{L^{\frac{ql}{l+q}}}\right).\\
		\end{aligned}
	\end{equation}
Then we need to deal with the two terms on the right-hand side of the inequality \eqref{2.5}, since
\begin{equation}\label{2.6}
	\begin{aligned}
	g^\varepsilon\nabla f^\varepsilon&=\int \frac{1}{\varepsilon^d}	g(y)\eta(\frac{x-y}{\varepsilon})dy \int \frac{1}{\varepsilon^d}f(y)\nabla \eta (\frac{x-y}{\varepsilon})\frac{1}{\varepsilon}dy\\
	&\leq C{\varepsilon}^{-1}\int_{B(x,\varepsilon)}\frac{1}{\varepsilon^d}|g(y)|dy \int_{B(x,\varepsilon)}\frac{1}{\varepsilon^d}|f(y)|dy\\
	&\leq C{\varepsilon}^{-1}\left(\int_{B(0,\varepsilon)}|g(x-z)|\frac{{\bf{1}}|_{B(0,\varepsilon)}(z)}{\varepsilon^d}dz\right)\left(\int_{B(0,\varepsilon)}|f(x-z)|\frac{{\bf{1}}|_{B(0,\varepsilon)}(z)}{\varepsilon^d}dz\right)\\
	&\leq C\varepsilon^{-1}\left(|g|*J_{1\varepsilon}(x)\right)\left(|f|*J_{1\varepsilon}(x)\right),
	\end{aligned}
\end{equation}
where $J_{1\varepsilon}=\frac{{\bf{1}}|_{B(0,\varepsilon)}(z)}{\varepsilon^d}\geq 0$ and $\int_{\mathbb{R}^d} J_{1\varepsilon}(z) dz=$measure of $B(0,1)$.
Then it follows from the H\"older's inequality and Minkowski's inequality, one has
\begin{equation}\label{2.7}
	\begin{aligned}
	\|g^\varepsilon \nabla f^\varepsilon\|_{L^{\frac{ql}{l+q}}}&\leq C\varepsilon^{-1}\|	\left(|g|*J_{1\varepsilon}(x)\right)\left(|f|*J_{1\varepsilon}(x)\right)\|_{L^{\frac{ql}{l+q}}}\\
	&\leq C\varepsilon^{-1}\||g|*J_{1\varepsilon}\|_{L^{l}}\||f|*J_{1\varepsilon}\|_{L^q}\\
	&\leq C\varepsilon^{-1}\|g\|_{L^{l}}\|f\|_{L^q}.
	\end{aligned}
\end{equation}
Similarly, we also have
\begin{equation}\label{2.8}
	\begin{aligned}
		\B\|f^\varepsilon \nabla g^\varepsilon\B\|_{L^{\frac{ql}{l+q}}}\leq C\varepsilon^{-1}\|f\|_{L^q}\|g\|_{L^{l}}.
	\end{aligned}
\end{equation}
Substituting \eqref{2.7} and \eqref{2.8} into \eqref{2.5}, one can obtain
\begin{equation}\label{2.9}
	\begin{aligned}
		\B\|\nabla \left(\frac{f^\varepsilon}{g^\varepsilon}\right)\B\|_{L^{\frac{ql}{l+q}}}\leq C\varepsilon^{-1}\|f\|_{L^q}\|g\|_{L^{l}}.
	\end{aligned}
\end{equation}
It follows from H\"older's inequality that
\begin{equation}\label{2.10}
\B\|\nabla \left(\frac{f^\varepsilon}{g^\varepsilon}\right)\B\|_{L^{\frac{pk}{k+p}}(0,T;L^{\frac{ql}{l+q}})}\leq C\varepsilon^{-1}\|f\|_{L^p(0,T;L^q)}\|g\|_{L^k(0,T;L^{l})}.
\end{equation}
Furthermore, if $1\leq q, l<\infty$, let $\{f_n\}, \{g_n\}\in C_{0}^\infty (\mathbb{T}^d)$ with $f_n\rightarrow f,\ g_n\rightarrow g$ strongly in $L^q(\mathbb{T}^d)$ and $L^{l}(\mathbb{T}^d)$, respectively. Thus, by the density arguments, we find that
\begin{equation}\label{2.11}
	\begin{aligned}
		&\varepsilon\B\|\nabla \left(\frac{f^\varepsilon}{g^\varepsilon}\right)\B\|_{L^{\frac{ql}{l+q}}}\leq C\varepsilon\left(\B\|\nabla \left(\frac{(f-f_n)^\varepsilon}{g^\varepsilon}\right)\B\|_{L^{\frac{ql}{l+q}}}+\B\|\nabla \left(\frac{f_n^\varepsilon}{g^\varepsilon}\right)\B\|_{L^{\frac{ql}{l+q}}}\right)\\
		\leq C&\varepsilon\left(\varepsilon^{-1}\|f-f_n\|_{L^q}\|g\|_{L^{l}}+\|g^\varepsilon \nabla f_n^\varepsilon\|_{L^{\frac{ql}{l+q}}}+\|f_n^\varepsilon \nabla (g-g_n)^\varepsilon\|_{L^{\frac{ql}{l+q}}}+\|f_n^\varepsilon \nabla g_n^\varepsilon\|_{L^{\frac{ql}{l+q}}}\right)\\
		\leq C & \left(\|f-f_n\|_{L^q}\|g\|_{L^{l}}+\varepsilon\|g\|_{L^{l}}\|\nabla f_n\|_{L^q}+\|f_n\|_{L^q}\|g-g_n\|_{L^{l}}+\varepsilon\|f_n\|_{L^q}\|\nabla g_n\|_{L^{l}}\right),
	\end{aligned}
\end{equation}
which gives
\begin{equation}\label{2.12}
	\begin{aligned}
	\limsup_{\varepsilon\rightarrow 0}&\varepsilon \B\|\nabla \left(\frac{f^\varepsilon}{g^\varepsilon}\right)\B\|_{L^{\frac{ql}{l+q}}}\\
\leq	& C\left(\|f-f_n\|_{L^q}\|g\|_{L^{l}}++\|f_n\|_{L^q}\|g-g_n\|_{L^{l}}\right)\rightarrow 0,\ as\ n\rightarrow +\infty.
	\end{aligned}
\end{equation}
Then taking $L^{\frac{pk}{k+p}}$ norm with respect to $t$ on \eqref{2.12} and using the H\"older's inequality, it leads to \eqref{2.4}.
\end{proof}
 \begin{lemma}\label{lem2.3}
		Assume that $0<c\leq \varrho (x,t)\in L^{l}(\mathbb{T}^d)$ and $\nabla v\in L^q(\mathbb{T}^d)$ with $l\geq \frac{2q}{q-1},\ 1\leq q\leq \infty$. Then
		\begin{equation}\label{b1}
\B\|\partial\left(\frac{(\varrho v)^\varepsilon}{\varrho ^\varepsilon}\right)\B\|_{L^{\frac{ql}{2q+l}}(\mathbb{T}^d)}\leq C\|\nabla v\|_{L^q(\mathbb{T}^d)}\left(\|\varrho\|_{L^{\frac{l}{2}}(\mathbb{T}^d)}+\|\varrho\|_{L^{l}(\mathbb{T}^d)}^2\right).
		\end{equation}
	\end{lemma}
\begin{proof}
	By direct computation, one has
	\begin{equation}\label{b3}
	\partial\left(\frac{(\varrho v)^\varepsilon}{\varrho ^\varepsilon}\right)=\frac{\partial (\varrho v)^\varepsilon-v\partial \varrho ^\varepsilon}{\varrho^\varepsilon}-\frac{\left((\varrho v)^\varepsilon-\varrho ^\varepsilon v\right)\partial \varrho ^\varepsilon}{(\varrho ^\varepsilon)^2}:=I_1+I_2.
	\end{equation}	
	Let $B(x,\varepsilon)=\{y\in \mathbb{T}^d; |x-y|<\varepsilon\}$, then Using the H\"older's inequality, we have
	\begin{equation}\label{b4}
	\begin{aligned}
	|I_1|&\leq C|\int \varrho(y)\left(v(y)-v(x)\right)\nabla_x\eta_\varepsilon(x-y)dy|\\
	&\leq C|\int_{\mathbb{R}^d}\varrho(y)\frac{v(y)-v(x)}{\varepsilon}\frac{1}{\varepsilon^d}\nabla \eta (\frac{x-y}{^\varepsilon})dy|\\
	&\leq C\left(\frac{1}{\varepsilon^d}\int _{B(x,\varepsilon)}\frac{|v(y)-v(x)|^{s_1}}{\varepsilon^{s_1}}dy\right)^{\frac{1}{s_1}}\left(\frac{1}{\varepsilon^d}\int_{B(x,\varepsilon)}|\varrho(y)|^{s_2}dy\right)^{\frac{1}{s_2}},\\
	\end{aligned}
	\end{equation}	
where $s_1\leq q,\ 2s_2\leq l$ and $\frac{1}{s_1}+\frac{1}{s_2}=1$.

	Using the mean value theorem, one has
	\begin{equation}\label{b5}
	\begin{aligned}
	\frac{1}{\varepsilon^d}\int_{B(x,\varepsilon)}\frac{|v(y)-v(x)|^{s_1}}{\varepsilon^{s_1}}dy&\leq C\frac{1}{\varepsilon^d}\int_{B(x,\varepsilon)}\int_0^1 |\nabla v(x+(y-x)s)|^{s_1}\frac{|y-x|^{s_1}}{\varepsilon^{s_1}}dsdy\\
	&\leq C \int_0^1\int_{B(0,1)}|\nabla v(x+s\varepsilon \omega)|^{s_1}d\omega ds\\
	&\leq C\int _{\mathbb{R}^d}|\nabla v(x-z)|^{s_1}\int_0^1\frac{{\bf{1}}_{B(0,\varepsilon s)}(z)}{(\varepsilon s)^d} ds dz\\
	&=|\nabla v|^{s_1} *J_\varepsilon (x),
	\end{aligned}
	\end{equation}	
	where $J_\varepsilon(z)=\int_0^1 \frac{{\bf{1}}_{B (0,\varepsilon s)}(z)}{(\varepsilon s)^d} ds dz\geq 0$ and it's easy to check that $\int_{\mathbb{R}^d} J_\varepsilon dz=measure\  of\  (B(0,1))$.
	Similarly, we also have
	\begin{equation}\label{b5-1}
		\begin{aligned}
			\frac{1}{\varepsilon^d}\int_{B(x,\varepsilon)}|\varrho(y)|^{s_2}dy&\leq C\int_{B(0,\varepsilon)}|\varrho(x-z)|^{s_2} dz\\
			&\leq C\int_{\mathbb{R}^d}|\varrho(x-z)|^{s_2}\frac{{\bf{1}}_{B (0,\varepsilon )}(z)}{(\varepsilon )^d}dz\\
			&\leq C|\varrho|^{s_2}*J_{1\varepsilon}(x),
		\end{aligned}
	\end{equation}
where $J_{1\varepsilon}=\frac{\bf{1}_{B(0,\varepsilon )}(z)}{(\varepsilon )^d}\geq 0$ and $\int_{\mathbb{R}^d} J_{1\varepsilon}(z)dz =measure\ of \  (B(0,1))$.

	Next, to estimate $I_2$, due to the H\"older's inequality, one deduces
	\begin{equation}\label{b6}
	\begin{aligned}
	|I_2|&=|\int \varrho(y)\left(v(y)-v(x)\right)\eta_\varepsilon(x-y)dy \frac{\int \varrho(y)\nabla_x \eta_\varepsilon(x-y)dy}{\left(\int \varrho(y)\eta_\varepsilon(x-y)dy\right)^2}|\\
	&\leq C\int_{B(x,\varepsilon)} \varrho(y) |v(y)-v(x)|\frac{1}{\varepsilon^d}dy\int_{B(x,\varepsilon)} \frac{1}{\varepsilon^d}\varrho (y)|\nabla \eta(\frac{x-y}{\varepsilon})|\frac{1}{\varepsilon} dy\\
	&\leq C \left(\frac{1}{\varepsilon^d}\int _{B(x,\varepsilon)}\frac{|v(y)-v(x)|^{s_1}}{\varepsilon^{s_1}}dy\right)^{\frac{1}{s_1}}\left(\frac{1}{\varepsilon^d}\int_{B(x,\varepsilon)}|\varrho(y)|^{s_2}dy \right)^{\frac{2}{s_2}}.
	\end{aligned}
	\end{equation}
	Therefor, by the same arguments as in \eqref{b5} and \eqref{b5-1}, in combination with \eqref{b3}-\eqref{b6}, we have
	\begin{equation}\label{b7}
	|I_1|+|I_2|\leq C\left(\nabla v|^{s_1} *J_\varepsilon (z)\right)^{\frac{1}{s_1}}\left(\left(\varrho^{s_2}*J_{1\varepsilon}\right)^{\frac{1}{s_2}}+\left(\varrho^{s_2}*J_{1\varepsilon}\right)^{\frac{2}{s_2}}\right).
	\end{equation}	
	Then from the Young's inequality, we arrive at
	\begin{equation}
	\begin{aligned}
	&\B\|\partial \frac{(\varrho v)^{\varepsilon}}{\varrho ^\varepsilon}\B\|_{{L^\frac{ql}{2q+l}}(\mathbb{T}^d)}\\
	\leq& C\|\left(|\nabla v|^{s_1}*J_\varepsilon(z)\right)^{\frac{1}{s_1}}\|_{L^{q}}\left(\|\left(|\varrho(y)|^{s_2}*J_{1\varepsilon}\right)^{\frac{1}{s_2}}\|_{L^{\frac{l}{2}}}+\|\left(|\varrho(y)|^{s_2}*J_{1\varepsilon}\right)^{\frac{2}{s_2}}\|_{L^{\frac{l}{2}}}\right)\\
	\leq& C\|\nabla v\|_{L^q}\|J_\varepsilon\|_{L^1}^{\frac{1}{s_1}}\left(\|\varrho\|_{L^{\frac{l}{2}}}\|J_{1\varepsilon}\|_{L^1}^{\frac{1}{s_2}}+\|\varrho \|_{L^{l}}^{2}\|J_{1\varepsilon}\|_{L^1}^{\frac{2}{s_1}}\right)\\
	\leq &  C\|\nabla v\|_{L^q}\left(\|\varrho\|_{L^{\frac{l}{2}}}+\|\varrho\|_{L^{l}}^2\right).
	\end{aligned}
	\end{equation}
	Then we have completed the proof of lemma \ref{lem2.3}.
	\end{proof}
On the other hand, if we let $\eta$ be non-negative   smooth function supported in the space-time ball of radius 1 and its integral equals to 1, we define the rescaled space-time mollifier $\eta_{\varepsilon}(t,x)=\frac{1}{\varepsilon^{d+1}}\eta(\frac{t}{\varepsilon},\frac{x}{\varepsilon})$ and
$$
f^{\varepsilon}(t,x)=\int_{0}^{T}\int_{\mathbb{T}^d}f(\tau,y)\eta_{\varepsilon}(t-\tau,x-y)dyd\tau.
$$
We list two lemmas for the proof of Theorem \eqref{the1.5}.
The fist one  is the  Lions type commutators on space-time mollifying kernel.
The second one is  the generalized Aubin-Lions Lemma, which helps us  to extend the energy equality up
to the initial time.
\begin{lemma} (\cite{[YWW],[LV],[CY]})
	\label{pLions}Let $1\leq p,q,p_1,q_1,p_2,q_2\leq \infty$,  with $\frac{1}{p}=\frac{1}{p_1}+\frac{1}{p_2}$ and $\frac{1}{q}=\frac{1}{q_1}+\frac{1}{q_2}$.
	Let $\partial$ be a partial derivative in space or time, in addition, let  $\partial_t f,\ \nabla f \in L^{p_1}(0,T;L^{q_1}(\Omega))$, $g\in L^{p_2}(0,T;L^{q_2}(\Omega))$.   Then, there holds $$\|{\partial(fg)^\varepsilon}-\partial(f\,{g}^\varepsilon)\|_{L^p(0,T;L^q(\Omega))}\leq C\left(\|\partial_{t} f\|_{L^{p_{1}}(0,T;L^{q_{1}}(\Omega))}+\|\nabla f\|_{L^{p_{1}}(0,T;L^{q_{1}}(\Omega))}\right)\|g\|_{L^{p_{2}}(0,T;L^{q_{2}}(\Omega))},
	$$
	for some constant $C>0$ independent of $\varepsilon$, $f$ and $g$. Moreover, $${\partial{(fg)^\varepsilon}}-\partial{(f\,{g^\varepsilon})}\to 0\quad\text{ in } {L^{p}(0,T;L^{q}(\Omega))},$$
	as $\varepsilon\to 0$ if $p_2,q_2<\infty.$
\end{lemma}
\begin{lemma}[\cite{[Simon]}]\label{AL}
	Let $X\hookrightarrow B\hookrightarrow Y$ be three Banach spaces with compact imbedding $X \hookrightarrow\hookrightarrow Y$. Further, let there exist $0<\theta <1$ and $M>0$ such that
	\begin{equation}\label{le1}
	\|v\|_{B}\leq M\|v\|_{X}^{1-\theta}\|v\|_{Y}^\theta\ \ for\ all\ v\in X\cap Y.\end{equation}
Denote for $T>0$,
\begin{equation}\label{le2}
	W(0,T):=W^{s_0,r_0}((0,T), X)\cap W^{s_1,r_1}((0,T),Y)
\end{equation}
with
\begin{equation}\label{le3}
	\begin{aligned}
		&s_0,s_1 \in \mathbb{R}; \ r_0, r_1\in [1,\infty],\\
		s_\theta :=(1-\theta)s_0&+\theta s_1,\ \f{1}{r_\theta}:=\f{1-\theta}{r_0}+\f{\theta}{r_1},\ s^{*}:=s_\theta -\f{1}{r_\theta}.
	\end{aligned}
\end{equation}
Assume that $s_\theta>0$ and $F$ is a bounded set in $W(0,T)$. Then, we have

If $s_{*}\leq 0$, then $F$ is relatively compact in $L^p((0,T),B)$ for all $1\leq p< p^{*}:=-\f{1}{s^{*}}$.

If $s_{*}> 0$, then $F$ is relatively compact in $C((0,T),B)$.

\end{lemma}

\section{Energy equality in compressible Navier-Stokes equations without vacuum}

When we consider the compressible Navier-Stokes equations with the density containing non-vacuum, due to the moumentum eqaution $\eqref{INS}_2$, just spatially regularized velocity $v^\varepsilon$ can be used as a test function to generate the global energy equality. However, to avoid a commutator estimate involing time $t$, we choose $\frac{(\varrho v)^\varepsilon}{\varrho ^\varepsilon}$ instead of $v^\varepsilon$ as the test function, which was introduced in \cite{[NNT],[LS]}. Hence, in this section,  we let  $\eta$ be non-negative   smooth function only supported in the space ball of radius 1 and its integral equals to 1. We define the rescaled space mollifier $\eta_{\varepsilon}(x)=\frac{1}{\varepsilon^{d}}\eta(\frac{x}{\varepsilon})$ and
$$
f^{\varepsilon}(x)=\int_{\Omega}f(y)\eta_{\varepsilon}(x-y)dyds.
$$

 \begin{proof}[Proof of Theorem \ref{the1.2}]
 	Multiplying $\eqref{INS}_2$ by $\left(\frac{(\varrho v)^\varepsilon}{\varrho^\varepsilon}\right)^\varepsilon$, then integrating over $(s,t]\times \mathbb{T}^d$ with $0<s<t<T$, we have
\begin{equation}\label{c1} \begin{aligned}
 \int_s^t\int\f{(\varrho v)^{\varepsilon}}{\varrho^{\varepsilon}}\B[\partial_{\tau}(\varrho v)^{\varepsilon}+ \Div(\varrho v\otimes v)^{\varepsilon}+\nabla  P(\varrho)^\varepsilon-\mu \Delta v^\varepsilon-(\mu+\lambda) \nabla(\Div v)^\varepsilon\B]=0.
 \end{aligned}\end{equation}
We will rewrite every term of the last  equality  to pass the limit of $\varepsilon$. For the first term, a  straightforward  calculation and $\eqref{INS}_{1}$ yields that
\begin{equation}\label{c2}
\begin{aligned} \int_s^t\int\f{(\varrho v)^{\varepsilon}}{\varrho^{\varepsilon}} \partial_{\tau}\B(\varrho v\B)^{\varepsilon}&= \int_s^t\int\f12\partial_{\tau}(\f{|(\varrho v)^{\varepsilon}|^{2}}{\varrho^{\varepsilon}})+\f12\partial_{\tau}\varrho^{\varepsilon}
\f{|(\varrho v)^{\varepsilon}|^{2}}{(\varrho^{\varepsilon})^{2}}\\
&= \int_s^t\int\f12\partial_{\tau}\B(\f{|(\varrho v)^{\varepsilon}|^{2}}{\varrho^{\varepsilon}}\B)- \f12\Div(\varrho v)^{\varepsilon}
\f{|(\varrho v)^{\varepsilon}|^{2}}{(\varrho^{\varepsilon})^{2}}.\\
 \end{aligned}\end{equation}
Integration by parts means that
\begin{equation}\label{c3}\begin{aligned}
  &\int_s^t\int\f{(\varrho v)^{\varepsilon}}{\varrho^{\varepsilon}}  \Div(\varrho v\otimes v)^{\varepsilon}\\
  =&-\int_s^t\int\nabla\B(\f{(\varrho v)^{\varepsilon}}{\varrho^{\varepsilon}} \B) [(\varrho v\otimes v)^{\varepsilon}-(\varrho v)^{\varepsilon}\otimes v^{\varepsilon}]-\int_s^t\int\nabla\B(\f{(\varrho v)^{\varepsilon}}{\varrho^{\varepsilon}} \B)(\varrho v)^{\varepsilon}\otimes v^{\varepsilon}.
 \end{aligned}\end{equation}
 Making use of integration by parts once again, we infer that
 \begin{equation}\label{c4}\begin{aligned}
&- \int_s^t\int\nabla\B(\f{(\varrho v)^{\varepsilon}}{\varrho^{\varepsilon}} \B)(\varrho v)^{\varepsilon}\otimes v^{\varepsilon}\\=& \int_s^t\int \Div v^{\varepsilon} \f{|(\varrho v)^\varepsilon|^{2}}{\varrho^{\varepsilon}}+\f12 \f{ v^{\varepsilon}}{\varrho^{\varepsilon}}\nabla|(\varrho v)^{\varepsilon}  |^{2}\\
=& \int_s^t\int \f12 \Div v^{\varepsilon} \f{|(\varrho v)^\varepsilon|^{2}}{\varrho^{\varepsilon}}-\f12 v^{\varepsilon}\nabla({\f{1}{\varrho^{\varepsilon}}} )|(\varrho v)^{\varepsilon}  |^{2}\\
=& \f12\int_s^t\int \Div( \varrho^{\varepsilon}v^{\varepsilon} ) \f{|(\varrho v)^\varepsilon|^{2}}{(\varrho^{\varepsilon})^{2}} \\
 =& \f12\int_s^t\int \Div\B[ \varrho^{\varepsilon}v^{\varepsilon}-(\varrho v)^{\varepsilon} \B] \f{|(\varrho v)^\varepsilon|^{2}}{(\varrho^{\varepsilon})^{2}} +
 \f12\int_s^t\int \Div (\varrho v)^{\varepsilon}   \f{|( \varrho v)^\varepsilon|^{2}}{(\varrho^{\varepsilon})^{2}}\\
 =&- \int_s^t\int\B[ \varrho^{\varepsilon}v^{\varepsilon}-(\varrho v)^{\varepsilon} \B] \f{(\varrho v)^\varepsilon}
 {\varrho^{\varepsilon}}\nabla\f{(\varrho v)^\varepsilon}{\varrho^{\varepsilon}} +
 \f12\int_s^t\int \Div (\varrho v)^{\varepsilon}   \f{|(\varrho v)^\varepsilon|^{2}}{(\varrho^{\varepsilon})^{2}}.
 \end{aligned}\end{equation}
Inserting  \eqref{c4}  into \eqref{c3}, we  arrive at
 \begin{equation}\label{c5}\begin{aligned}
  &\int_s^t\int\f{(\varrho v)^{\varepsilon}}{\varrho^{\varepsilon}}  \Div(\varrho v\otimes v)^{\varepsilon}\\=&-\int_s^t\int\nabla\B(\f{(\varrho v)^{\varepsilon}}{\varrho^{\varepsilon}} \B) [(\varrho v\otimes v)^{\varepsilon}-(\varrho v)^{\varepsilon}\otimes v^{\varepsilon}]\\&-\int_s^t\int\B[ \varrho^{\varepsilon}v^{\varepsilon}-(\varrho v)^{\varepsilon} \B] \f{(\varrho v)^\varepsilon}
  {\varrho^{\varepsilon}}\nabla \f{(\varrho v)^\varepsilon}{\varrho^{\varepsilon}}+
 \f12\int_s^t\int \Div (\varrho v)^{\varepsilon}   \f{|(\varrho v)^\varepsilon|^{2}}{(\varrho^{\varepsilon})^{2}}.
   \end{aligned}\end{equation}
   For the pressure term, by the integration by parts, one has
    \begin{equation}\label{c6}
    \begin{aligned}
 \int_s^t\int\f{(\varrho v)^{\varepsilon}}{\varrho^{\varepsilon}} \nabla(P(\varrho))^{\varepsilon}= &\int_s^t\int\frac{(\varrho v)^\varepsilon}{\varrho^\varepsilon}\nabla \left[(P(\varrho))^\varepsilon-P(\varrho^\varepsilon)\right]+\int_s^t\int  \f{(\varrho v)^{\varepsilon}}{\varrho^{\varepsilon}} \nabla  P(\varrho^{\varepsilon})\\
 =&-\int_s^t\int\Div\B[\f{(\varrho v)^{\varepsilon}}{\varrho^{\varepsilon}} \B] [(P(\varrho))^{\varepsilon}- P(\varrho^{\varepsilon}) ]+\int_s^t\int  \f{(\varrho v)^{\varepsilon}}{\varrho^{\varepsilon}} \nabla  P(\varrho^{\varepsilon}).
 \end{aligned} \end{equation}
Using the mass equation $\eqref{INS}_1$, the second term on the right hand-side of \eqref{c6} can be rewritten as
\begin{equation}\label{c7}
\begin{aligned}
\int_s^t\int \frac{(\varrho v)^\varepsilon}{\varrho ^\varepsilon}\nabla P(\varrho^\varepsilon)&=\int_s^t\int (\varrho v)^\varepsilon \gamma (\varrho^\varepsilon)^{\gamma-2}\nabla \varrho^\varepsilon=\int_s^t\int (\varrho v)^\varepsilon\frac{\gamma}{\gamma -1}\nabla (\varrho^\varepsilon)^{\gamma-1}\\
&=\int_s^t\int \partial_\tau\varrho^\varepsilon \frac{\gamma}{\gamma-1}(\varrho^\varepsilon)^{\gamma-1}dxd\tau=\int_s^t\int \frac{1}{\gamma-1}\partial_\tau P(\varrho^\varepsilon).
\end{aligned}
\end{equation}

   It is clear that
 \begin{equation}\label{c8}\begin{aligned} &-\mu\int_s^t\int\f{(\varrho v)^{\varepsilon}}{\varrho^{\varepsilon}}  \Delta v^{\varepsilon}= \mu\int_s^t\int-\Delta v ^\varepsilon v^\varepsilon - \Delta v^\varepsilon
 \f{(\varrho v)^{\varepsilon}-\varrho^{\varepsilon} v^{\varepsilon}}{\varrho^{\varepsilon}},
 \\
&-(\mu+\lambda)\int_s^t\int\f{(\varrho v)^{\varepsilon}}{\varrho^{\varepsilon}}\nabla(\Div v)^\varepsilon=(\mu+\lambda) \int_s^t\int -\nabla  (\Div v)^\varepsilon  v^{\varepsilon}-\nabla(\Div v)^{\varepsilon}
 \f{(\varrho v)^{\varepsilon}-\varrho^{\varepsilon} v^{\varepsilon}}{\varrho^{\varepsilon}}.
\end{aligned}  \end{equation}
Substituting \eqref{c2}, \eqref{c5}-\eqref{c8} into \eqref{c1}, we see that
 \begin{equation}\label{c9}
 \begin{aligned}
&\int_s^t\int\partial_{\tau}\left(\f12\f{|(\varrho v)^{\varepsilon}|^{2}}{\varrho^{\varepsilon}}+\frac{1}{\gamma-1}P(\varrho^\varepsilon)\right) +\mu \int_s^t\int |\nabla v^{\varepsilon} |^2 +(\mu+\lambda)\int_s^t\int |\nabla \Div v^{\varepsilon} |^2\\
=
 &\int_s^t\int \mu \Delta v^{\varepsilon}
 \f{(\varrho v)^{\varepsilon}-\varrho^{\varepsilon} v^{\varepsilon}}{\varrho^{\varepsilon}}+(\mu+\lambda)\int_s^t\int\nabla(\Div v)^{\varepsilon}
 \f{(\varrho v)^{\varepsilon}-\varrho^{\varepsilon} v^{\varepsilon}}{\varrho^{\varepsilon}}
\\&\ \ \ +\int_s^t\int\Div\B[\f{(\varrho v)^{\varepsilon}}{\varrho^{\varepsilon}} \B] [(P(\varrho))^{\varepsilon}- P(\varrho^{\varepsilon}) ]\\
&+\int_s^t\int\nabla\B(\f{(\varrho v)^{\varepsilon}}{\varrho^{\varepsilon}} \B) [(\varrho v\otimes v)^{\varepsilon}-(\varrho v)^{\varepsilon}\otimes v^{\varepsilon}] +\int_s^t\int\B[ \varrho^{\varepsilon}v^{\varepsilon}-(\varrho v)^{\varepsilon} \B] \f{(\varrho v)^\varepsilon}
{\varrho^{\varepsilon}}\nabla\f{(\varrho v)^\varepsilon}{\varrho^{\varepsilon}}.
  \end{aligned}\end{equation}
  Next, we need to prove that the terms on the right hand-side of \eqref{c9} tend to zero as $\varepsilon\rightarrow 0$.

Under the hypothesis
\begin{equation}\label{unified}
	\begin{aligned}
&0<  c\leq\varrho \in L^{ k}(0,T;L^{ l}(\mathbb{T}^d)), \\
& v\in L^{p }(0,T;L^{q} (\mathbb{T}^{d})),\nabla v\in L^{\f{kp}{k(p-2)-4p}}(0,T; L^{\f{lq}{l(q-2)-4q}}(\mathbb{T}^{d})),
\end{aligned}
\end{equation}
with $k> \max\{\frac{4p}{p-2},\f{(\gamma-2)p }{2}\}$, $l> \max\{\frac{4q}{q-2}, \f{(\gamma-2)q}{2}\}$ and $k(4-p)+7p\geq 0,\  l(4-q)+7q\geq 0$.

It follows from  Lemma \ref{lem2.1}  that
\begin{equation}\label{ec10}\begin{aligned}
		  \|\Delta v^\varepsilon\|_{L^{\f{kp}{k(p-2)-4p}}(L^{\f{lq}{l(q-2)-4q}})}
\leq& C\|\Div (\nabla v)^\varepsilon\|_{L^{\f{kp}{k(p-2)-4p}}(L^{\f{lq}{l(q-2)-4q}})}\\
\leq& C\varepsilon^{-1}\| \nabla v \|_{L^{\f{kp}{k(p-2)-4p}}(L^{\f{lq}{l(q-2)-4q}})}, \end{aligned}\end{equation}
and
	\be\ba	 \limsup_{\varepsilon\rightarrow0}\varepsilon\|\Delta v^\varepsilon\|&_{L^{\f{kp}{k(p-2)-4p}}(L^{\f{lq}{l(q-2)-4q}})}=0.
\end{aligned}\end{equation}
Using the  Constantin-E-Titi type   commutators on mollifying kernel  Lemma  \ref{lem2.2},
we know that
 \begin{equation}\label{ec11}\begin{aligned}
			\|(\varrho v)^\varepsilon-\varrho^\varepsilon v^\varepsilon\|_{L^{\frac{kp}{2k+4p}}(L^{\frac{lq}{2l+4q}})}\leq C\varepsilon \|\varrho \|_{L^{\frac{kp}{k(4-p)+8p}}(L^{\frac{lq}{l(4-q)+8q}} )} \|\nabla v\|_{L^{\f{kp}{k(p-2)-4p}}(L^{\f{lq}{l(q-2)-4q}})}.
\end{aligned}\end{equation}
Combining the the H\"older's inequality and  \eqref{ec10}-\eqref{ec11},
we arrive at
\begin{equation}\label{c11}\begin{aligned}
&\B|\int_s^t\int \mu \Delta v^\varepsilon
 \f{(\varrho v)^{\varepsilon}-\varrho^{\varepsilon} v^{\varepsilon}}{\varrho^{\varepsilon}}\B|\\
\leq& C\|\Delta v^\varepsilon\|_{L^{\f{kp}{k(p-2)-4p}}(L^{\f{lq}{l(q-2)-4q}})}\B\|
 \f{(\varrho v)^{\varepsilon}-\varrho^{\varepsilon} v^{\varepsilon}}{\varrho^{\varepsilon}}\B\|_{L^{\frac{kp}{2k+4p}}(L^{\frac{lq}{2l+4q}})}\\
\leq& C\varepsilon\|\Delta v^\varepsilon\|_{L^{\f{kp}{k(p-2)-4p}}(L^{\f{lq}{l(q-2)-4q}})}\|\varrho \|_{L^{\frac{kp}{k(4-p)+8p}}(L^{\frac{lq}{l(4-q)+8q}}) } \|\nabla v\|_{L^{\f{kp}{k(p-2)-4p}}(L^{\f{lq}{l(q-2)-4q}})},
 \end{aligned}\end{equation}
where we need $ k >\frac{4p}{p-2},\ k(4-p)+7p\geq 0$ and $l >\frac{4q}{q-2},\ l(4-q)+7q\geq 0$.

 As a consequence, we get
$$ \limsup_{\varepsilon\rightarrow0}\B|\mu\int_s^t\int\Delta v^\varepsilon
 \f{(\varrho v)^{\varepsilon}-\varrho^{\varepsilon} v^{\varepsilon}}{\varrho^{\varepsilon}}\B|=0.
$$
Likewise, there holds
\begin{equation}\label{c12}
\begin{aligned}
\limsup_{\varepsilon\rightarrow0}\B|(\mu+\lambda )\int_s^t\int\nabla(\Div v)^{\varepsilon}
 \f{(\varrho v)^{\varepsilon}-\varrho^{\varepsilon} v^{\varepsilon}}{\varrho^{\varepsilon}}\B|=0.
\end{aligned}\end{equation}
Applying Lemma  \ref{lem2.3}, we get
\be\label{ec3.16}
\|\Div\B[\f{(\varrho v)^{\varepsilon}}{\varrho^{\varepsilon}} \B]\|_{L^{\f{kp}{k(p-2)-2p}}(L^{\f{lq}{l(q-2)-2q}})}\leq C\|\nabla v\|_{L^{\f{kp}{k(p-2)-4p}}(L^{\f{lq}{l(q-2)-4q}})} \left(\|\varrho\|_{L^{\frac{k}{2}}(L^{\frac{l}{2}})}+\|\varrho \|_{L^k(L^{l})}^2\right)
\ee
For some $\theta\in [0,1]$, the mean value theorem, the H\"older's inequality and the triangle inequality  ensure that
\be\ba\label{ec3.17}
\|  \varrho^{\gamma}-( \varrho^{\varepsilon} )^{\gamma} \|_{L^\frac{ql}{2(l+q) }}
&=\|[\varrho+\theta(\varrho-\varrho^{\varepsilon}) ]^{\gamma-1} (\varrho -\varrho^{\varepsilon}) \|_{L^\frac{ql}{2(l+q) }}\\
&\leq C \| \varrho \|_{L^\frac{\gamma ql}{2(l+q) }}^{\gamma-1} \| \varrho -\varrho^{\varepsilon}  \|_{L^\frac{\gamma ql}{2(l+q) }},
\ea\ee
which follows from that, by Lemma \ref{lem2.11}, if $\varrho\in L^\frac{\gamma ql}{2(l+q) }$, as $\varepsilon\rightarrow 0$,
\be\ba\label{ec3.18}\|  \varrho^{\gamma}-( \varrho^{\gamma} )^{\varepsilon} \|_{L^\frac{ql}{2(l+q) }}\rightarrow 0.\ea\ee
With the help  of  the triangle inequality, the H\"older's inequality and \eqref{ec3.16}-\eqref{ec3.18}, we obtain
\begin{equation}\label{c13}\begin{aligned}  &\int_s^t\int\Div\B[\f{(\varrho v)^{\varepsilon}}{\varrho^{\varepsilon}} \B] [( \varrho^{\gamma} )^{\varepsilon}- (\varrho^{\varepsilon})^{\gamma} ]\\
  \leq& \int_s^t\int\Div\B[\f{(\varrho v)^{\varepsilon}}{\varrho^{\varepsilon}} \B] |( \varrho^{\gamma} )^{\varepsilon}- \varrho^{\gamma} |+ \int_s^t\int\Div\B[\f{(\varrho v)^{\varepsilon}}{\varrho^{\varepsilon}} \B] | \varrho^{\gamma} - (\varrho^{\varepsilon})^{\gamma}|\\
  \leq& C \|\Div\B[\f{(\varrho v)^{\varepsilon}}{\varrho^{\varepsilon}} \B]\|_{L^{\f{kp}{k(p-2)-2p}}(L^{\f{lq}{l(q-2)-2q}})}\B(\|( \varrho^{\gamma} )^{\varepsilon}- \varrho^{\gamma}\|_{L^\frac{pk}{2(k+p) }(L^\frac{ql}{2(l+q) })}+\|  \varrho^{\gamma}-( \varrho^{\gamma} )^{\varepsilon} \|_{L^\frac{pk}{2(k+p) }(L^\frac{ql}{2(q+l) })}\B)
  \\
  \leq&   C\|\nabla v\|_{L^{\f{kp}{k(p-2)-4p}}(L^{\f{lq}{l(q-2)-4q}})} \left(\|\varrho\|_{L^{\frac{l}{2}}(L^{\frac{k}{2}})}+\|\varrho \|_{L^l(L^{k})}^2\right)\\
  &\ \ \ \ \ \ \
\B(\|( \varrho^{\gamma} )^{\varepsilon}- \varrho^{\gamma}\|_{L^\frac{pk}{2(k+p) }(L^\frac{ql}{2(l+q) })}+\|  \varrho^{\gamma}-( \varrho^{\gamma} )^{\varepsilon} \|_{L^\frac{pk}{2(k+p) }(L^\frac{ql}{2(q+l) })}\B)
  \end{aligned}\end{equation}
  which means that
  $$
  \limsup_{\varepsilon\rightarrow0}\int_s^t\int\Div\left(\f{(\varrho v)^{\varepsilon}}{\varrho^{\varepsilon}} \right) \left((P(\varrho))^{\varepsilon}- P(\varrho^{\varepsilon}) \right)=0.$$
At this stage,  it is enough to show
\begin{equation}\label{c14}\begin{aligned}
\limsup_{\varepsilon\rightarrow0}\int_s^t\int\nabla\B(\f{(\varrho v)^{\varepsilon}}{\varrho^{\varepsilon}} \B) [(\varrho v\otimes v)^{\varepsilon}-(\varrho v)^{\varepsilon}\otimes v^{\varepsilon}] +\int_s^t\int\B[ \varrho^{\varepsilon}v^{\varepsilon}-(\varrho v)^{\varepsilon} \B] \f{(\varrho v)^\varepsilon}
 {\varrho^{\varepsilon}}\nabla\f{(\varrho v)^\varepsilon}{\varrho^{\varepsilon}} =0,
  \end{aligned}\end{equation}
In view of   \eqref{2.3} and  Lemma  \ref{lem2.1},
\be\label{c3.22}
\B\|\nabla\B(\f{(\varrho v)^{\varepsilon}}{\varrho^{\varepsilon}} \B)\B\|_{L^{\f{kp}{k+2p}}(L^{\f{lq}{l+2q}})}\leq C\varepsilon^{-1}\| v\|_{L^{p}(L^{q})}\|\varrho \|_{L^k (L^{l})}^2,
\ee
and
\be\label{c3.23}
\limsup_{\varepsilon\rightarrow0}\varepsilon\B\|\nabla\B(\f{(\varrho v)^{\varepsilon}}{\varrho^{\varepsilon}} \B)\B\|_{L^{\f{kp}{k+2p}}(L^{\f{lk}{l+2q}})}=0.
\ee
Taking advantage of  \eqref{fg'}     in Lemma \ref{lem2.2},
\be\label{c3.24}
\|(\varrho v\otimes v)^{\varepsilon}-(\varrho v)^{\varepsilon}\otimes v^{\varepsilon} \|_{L^{\f{kp}{k(p-1)-3p}}(L^{\f{lq}{l(q-1)-3q}})}
 \leq C\varepsilon\|\nabla v\|_{L^{\f{kp}{k(p-2)-4p}}(L^{\f{lq}{l(q-2)-4q}})}  \| v\|_{L^{p}(L^{q})}\|\varrho\|_{L^{k}(L^{l})}
\ee
Thanks to the H\"older's inequality and    \eqref{c3.22}-\eqref{c3.24}, we find
  \begin{equation}\label{c18}\begin{aligned}
 &\B|\int_s^t\int\nabla\B(\f{(\varrho v)^{\varepsilon}}{\varrho^{\varepsilon}} \B) [(\varrho v\otimes v)^{\varepsilon}-(\varrho v)^{\varepsilon}\otimes v^{\varepsilon}]\B|\\
\leq & C\B\|\nabla\B(\f{(\varrho v)^{\varepsilon}}{\varrho^{\varepsilon}} \B)\B\|_{L^{\f{kp}{k+2p}}(L^{\f{lk}{l+2q}})}\|(\varrho v\otimes v)^{\varepsilon}-(\varrho v)^{\varepsilon}\otimes v^{\varepsilon} \|_{L^{\f{kp}{k(p-1)-3p}}(L^{\f{lq}{l(q-1)-3q}})}\|1\|_{L^{k}(L^{l})}\\
\leq & C\varepsilon\B\|\nabla\B(\f{(\varrho v)^{\varepsilon}}{\varrho^{\varepsilon}} \B)\B\|_{L^{\f{kp}{k+2p}}(L^{\f{lk}{l+2q}})}\|\nabla v\|_{L^{\f{kp}{k(p-2)-4p}}(L^{\f{lq}{l(q-2)-4q}})}  \| v\|_{L^{p}(L^{q})}\|\varrho\|_{L^{k}(L^{l})}.
 \end{aligned}\end{equation}
 which in turn implies
$$
\limsup_{\varepsilon\rightarrow0}\B|\int_s^t\int\nabla\B(\f{(\varrho v)^{\varepsilon}}{\varrho^{\varepsilon}} \B) [(\varrho v\otimes v)^{\varepsilon}-(\varrho v)^{\varepsilon}\otimes v^{\varepsilon}]\B|=0.$$
We turn our attentions to the term $\int_s^t\int\B[ \varrho^{\varepsilon}v^{\varepsilon}-(\varrho v)^{\varepsilon} \B] \f{(\varrho v)^\varepsilon}
{\varrho^{\varepsilon}}\nabla\f{(\varrho v)^\varepsilon}{\varrho^{\varepsilon}}$.
We conclude from Lemma \ref{lem2.2} that
\begin{equation}\label{1.3.16}
\|\varrho^{\varepsilon}v^{\varepsilon}-(\varrho v)^{\varepsilon}\|_{L^{\f{kp}{k(p-2)-3p}}(L^{\f{lq}{l(q-2)-3q}})} \leq  C \varepsilon\|\nabla v\|_{L^{\f{kp}{k(p-2)-4p}}(L^{\f{lq}{l(q-2)-4q}})}  \|\varrho\|_{L^{k} (L^{l} )}.
\end{equation}
Using the H\"older's inequality  and \eqref{1.3.16}  we find,
\begin{equation}\label{1.3.17} \begin{aligned}
&\B|\int_s^t\int\B[ \varrho^{\varepsilon}v^{\varepsilon}-(\varrho v)^{\varepsilon} \B] \f{(\varrho v)^\varepsilon}
{\varrho^{\varepsilon}}\nabla \f{(\varrho v)^\varepsilon}{\varrho^{\varepsilon}}\B| \\
\leq&C  \|\varrho^{\varepsilon}v^{\varepsilon}-(\varrho v)^{\varepsilon}\|_{L^{\f{kp}{k(p-2)-3p}}(L^{\f{lq}{l(q-2)-3q}})}\B\|\f{(\varrho v)^\varepsilon}
{\varrho^{\varepsilon}}\B\|_{L^{\frac{kp}{k+p}} (L^{\frac{lq}{l+q}}) } \B\|\nabla\f{(\varrho v)^\varepsilon}{\varrho^{\varepsilon}}\B\|_{L^{\f{kp}{k+2p}}(L^{\f{lk}{l+2q}})}\\
\leq& C \varepsilon\|\nabla v\|_{L^{\f{kp}{k(p-2)-4p}}(L^{\f{lq}{l(q-2)-4q}})}  \|\varrho\|_{L^{k}( L^{l}) }^2 \|v\|_{L^{p}L^{q}}  \B\|\nabla\f{(\varrho v)^\varepsilon}{\varrho^{\varepsilon}}\B\|_{L^{\f{kp}{k+2p}}(L^{\f{lk}{l+2q}})}\\
\end{aligned}\end{equation}
Together this with \eqref{c3.22} and \eqref{c3.23} yield  that
$$
\limsup_{\varepsilon\rightarrow0}
\B|\int_s^t\int\B[ \varrho^{\varepsilon}v^{\varepsilon}-(\varrho v)^{\varepsilon} \B] \f{(\varrho v)^\varepsilon}
{\varrho^{\varepsilon}}\nabla \f{(\varrho v)^\varepsilon}{\varrho^{\varepsilon}}\B|=0.
$$
Collecting  all the above estimates, using the weak continuity of $\varrho$ and $\varrho v$, we   complete the proof of Theorem \ref{the1.2}.
\end{proof}
\section{Energy equality in compressible Navier-Stokes equations allowing vacuum}
Unlike the case with non-vacuum, when we consider the compressible Navier-Stokes equations with the density containing vacuum, only spatially regularized velocity $v^\varepsilon$ fails to possess enough temporal regularity to qualify for a test function. To over this difficulity, we need to mollify the velocity both in space and time. Hence, in this section,  we let  $\eta$ be non-negative   smooth function  supported in the space-time ball of radius 1 and its integral equals to 1. We define the rescaled space mollifier $\eta_{\varepsilon}(t,x)=\frac{1}{\varepsilon^{d+1}}\eta(\frac{t}{\varepsilon},\frac{x}{\varepsilon})$ and
$$
f^{\varepsilon}(t,x)=\int_0^T\int_{\mathbb{T}^d}f(\tau,y)\eta_{\varepsilon}(t-\tau,x-y)dyd\tau.
$$
 \begin{proof}[Proof of Theorem \ref{the1.2}]

Let $\phi(t)$ be a smooth function compactly supported in $(0,+\infty)$.
Multiplying $\eqref{INS}_2$ by $(\phi v^{\varepsilon})^\varepsilon$, then integrating over $(0,T)\times \mathbb{T}^d$ , we infer that
\begin{equation}\label{d1} \begin{aligned}
		\int_0^T\int \phi(t)v^{\varepsilon}\B[\partial_{t}(\varrho v)^{\varepsilon}+ \Div(\varrho v\otimes v)^{\varepsilon}+\nabla P(\varrho)^\varepsilon- \mu \Delta v^\varepsilon-(\mu+\lambda)\nabla( \Div v)^\varepsilon\B]=0.
\end{aligned}\end{equation}
 To pass the limit of $\varepsilon$, we
reformulate  every term of the last equation.
A  straightforward  computation leads to
\begin{equation}\label{d2}
	\begin{aligned}
		\int_0^T\int \phi(t)v^{\varepsilon} \partial_{t} (\varrho v )^{\varepsilon}=&\int_0^T\int \phi(t)v^{\varepsilon}\B[ \partial_{t} (\varrho v )^{\varepsilon}-\partial_{t}(\varrho v^{\varepsilon})\B]+ \int_0^T\int \phi(t)v^{\varepsilon} \partial_{t}(\varrho v^{\varepsilon}) \\
		=& \int_0^T\int \phi(t)v^{\varepsilon} \B[\partial_{t} (\varrho v )^{\varepsilon}-\partial_{t}(\varrho v^{\varepsilon})\B]+\int_0^T\int \phi(t)\varrho\partial_t{\frac{|v^{\varepsilon}|^2}{2}} \\
		&+\int_0^T\int \phi(t)\varrho_t|v^{\varepsilon}|^2.
\end{aligned}\end{equation}
It follows from integration by parts and the mass equation $\eqref{INS}_1$ that
\begin{align}
	&\int_0^T\int\phi(t)v^{\varepsilon} \Div(\varrho v\otimes v)^{\varepsilon}\nonumber\\
	=& \int_0^T\int\phi(t)  v^{\varepsilon}  \Div[(\varrho v\otimes v)^{\varepsilon}-(\varrho  v)\otimes v^{\varepsilon}]+\int_0^T\int\phi(t)v^{\varepsilon}\Div(\varrho  v\otimes v^{\varepsilon})\nonumber\\
	=& -\int_0^T\int\phi(t)  \nabla v^{\varepsilon}  [(\varrho v\otimes v)^{\varepsilon}-(\varrho  v)\otimes v^{\varepsilon}]+  \int_0^T\int \phi(t)\left(\Div (\varrho v ) |v^{\varepsilon}|^{2}+\f12 \varrho v \nabla|v^{\varepsilon}  |^{2}
	\right)\nonumber\\
	=& -\int_0^T\int\phi(t) \nabla  v^{\varepsilon}  [(\varrho v\otimes v)^{\varepsilon}-(\varrho  v)\otimes v^{\varepsilon}]+\f{1}{2}\int_0^T\int\phi(t) \Div (\varrho v ) |v^{\varepsilon}|^{2}\nonumber\\
	=& -\int_0^T\int\phi(t)  \nabla v^{\varepsilon}  [(\varrho v\otimes v)^{\varepsilon}-(\varrho  v)\otimes v^{\varepsilon}]-\frac{1}{2}\int_0^T\int \phi(t) \partial_t \varrho |v^{\varepsilon}|^{2}. \label{d3}
\end{align}
We rewrite the pressure term as
\begin{equation}\label{d6}
	\begin{aligned}
		&\int_0^T\int\phi(t)v^{\varepsilon}\nabla (\varrho^\gamma)^{\varepsilon}= \int_0^T\int\phi(t) [v^{\varepsilon}\nabla(\varrho^\gamma)^{\varepsilon}-v\nabla (\varrho^\gamma)]+\int_0^T\int\phi(t)v \nabla (\varrho^\gamma).
\end{aligned} \end{equation}
Then using the integration by parts and mass equation $\eqref{INS}_1$ again, we find
$$\ba
\int_0^T\int\phi(t)v \cdot\nabla (\varrho^\gamma)
=&-\int_0^T\int\phi(t) \varrho^{\gamma-1} \varrho\text{div\,}v\\
=&\int_0^T\int\phi(t) \varrho^{\gamma-1}(\partial_{t}\varrho+v\cdot\nabla\varrho)
\\
=&\f{1}{\gamma}\int_0^T\int\phi(t)  \partial_{t}\varrho^{\gamma } +\f{1}{\gamma}\int_0^T\int\phi(t)v\cdot\nabla\varrho^{\gamma },
\ea $$
which in turn means that
\begin{equation}\label{d62}\begin{aligned}
		\int_0^T\int\phi(t)v \cdot\nabla (\varrho^\gamma)
		= \f{1}{\gamma-1}\int_0^T\int\phi(t)  \partial_{t}\varrho^{\gamma }.
\end{aligned}\end{equation}
Thanks to integration by parts, we arrive at
\begin{equation}\label{d8}\begin{aligned} &-\mu\int_0^T\int \phi(t)v^{\varepsilon} \Delta v^{\varepsilon}=\mu \int_0^T\int \phi(t)|\nabla v^{\varepsilon}|^{2},\\
		&- (\mu+\lambda)\int_0^T\int \phi(t)v^{\varepsilon} \nabla\text{div\,}v^\varepsilon=(\mu+\lambda)\int_0^T\int \phi(t)|\Div{v^{\varepsilon}}|^{2}.
\end{aligned}  \end{equation}
Plugging  \eqref{d2}-\eqref{d8} into \eqref{d1} and using the integration by parts, we conclude  that
\begin{equation}\label{d9}
	\begin{aligned}
		&-\int_0^T\int \phi(t)_t\left(\varrho{\frac{|v^{\varepsilon}|^2}{2}}+\f{1}{\gamma-1}\varrho^{\gamma }\right)+\int_0^T\int \left(\mu |\nabla v^\varepsilon|^2+(\mu+\lambda)|\Div v^\varepsilon|^2\right)\\
		=&-\int_0^T\int \phi(t)v^{\varepsilon} \B[\partial_{t} (\varrho v )^{\varepsilon}-\partial_{t}(\varrho v^{\varepsilon})\B]+\int_0^T\int\phi(t) \nabla v^{\varepsilon}  [(\varrho v\otimes v)^{\varepsilon}-(\varrho  v)\otimes v^{\varepsilon}]\\ &-\int_0^T\int\phi(t) [v^{\varepsilon}\nabla(\varrho^\gamma)^{\varepsilon}-v\nabla (\varrho^\gamma)].
\end{aligned}\end{equation}
It is enough  to prove that the terms on the right hand-side of \eqref{d9} tend to zero as $\varepsilon\rightarrow 0$.

Under the hypothesis
\be
\ba
&0\leq \varrho\in L^{k}(0,T;L^{l}(\mathbb{T}^d))
,\nabla\sqrt{\varrho}\in L^{\f{2pk}{2k(p-3)-p}}(0,T;L^{\f{2lq}{2l(q-3)-q}}(\mathbb{T}^d))\\
& v\in L^{p }(0,T;L^{q}(\mathbb{T}^d)),\nabla v \in L^{\f{pk}{k(p-2)-p}}(0,T;L^{\f{ql}{l(q-2)-q}}(\mathbb{T}^d))\ and\ v_0\in L^{\max\{
	\frac{2\gamma }{\gamma -1},\frac{q}{2}\}}(\mathbb{T}^d),
\ea\ee
with $k> \max\{\frac{p}{2(p-3)},\frac{p}{p-2}, \frac{(\gamma-1)p}{2},\frac{(\gamma-1)(d+q)p}{2q-d(p-3)}\}$, $l> \max\{\frac{q}{2(q-3)},\frac{q}{q-2}, \frac{(\gamma-1)q}{2}\}$, $p>3$ and $q>\max\{3, \frac{d(p-3)}{2}\}$.

In view of H\"older's inequality and Lemma \eqref{pLions}, we know that
\begin{equation}\label{3.81}\begin{aligned}
\int_s^t\int \phi(t)v^{\varepsilon} &\B[\partial_{t} (\varrho v )^{\varepsilon}-\partial_{t}(\varrho v^{\varepsilon})\B]\leq C\|v^{\varepsilon}\|_{L^{p}(L^{q})}
\|\partial_{t} (\varrho v )^{\varepsilon}-\partial_{t}(\varrho v^{\varepsilon})\|_{L^{\f{ p }{p-1}}(L^\f{q}{q-1})}\\
&\leq C\|v\|^{2}_{L^{p}(L^{q})}
\left(\|\varrho_{t}\|_{L^{\f{ p }{p-2}}(L^\f{q}{q-2})}+\|\nabla \varrho\|_{L^{\f{ p }{p-2}}(L^\f{q}{q-2})}\right).
\end{aligned}\end{equation}
To bound $\varrho_{t}$ and $\nabla \varrho$, we
employ  mass equation $\eqref{INS}_1$ to obtain
$$
\varrho_t=-2\sqrt{\varrho}v\cdot\nabla\sqrt{\varrho}-{\varrho}\text{div}v, \ and\ \nabla \varrho=2\sqrt{\varrho}\nabla \sqrt{\varrho}. $$
As a consequence, the triangle inequality and H\"older's inequality  guarantee that
\begin{equation}\label{3.91}\begin{aligned}
&\|\varrho_{t}\|_{L^{\f{ p }{p-2}}(L^\f{q}{q-2})}\\\leq&
 C\left(\|-2\sqrt{\varrho}v\cdot\nabla\sqrt{\varrho} \|_{L^{\f{ p }{p-2}}(L^\f{q}{q-2})}+
\|{\varrho}\text{div}v\|_{L^{\f{ p }{p-2}}(L^\f{q}{q-2})}\right)\\
\leq& C\left(
\|v\|_{L^{p}L^{q}}\|\nabla\sqrt{\varrho} \|_{L^{\f{2pk}{2k(p-3)-p}}(L^{\f{2lq}{2l(q-3)-q}})}\|\varrho\|_{L^{k}(L^{l})}^{1/2}+\| \nabla v\|_{L^{\f{pk}{k(p-2)-p}}(L^{\f{ql}{l(q-2)-q}})} \|\varrho\|_{L^{k}(L^{l})}\right),
 \end{aligned}\end{equation}
 and
\begin{equation}\label{3.92}
	\begin{aligned}
	\|\nabla \varrho\|_{L^{\f{ p }{p-2}}(L^\f{q}{q-2})}\leq & C\|\sqrt{\varrho}\nabla \sqrt{\varrho}\|_{L^{\f{ p }{p-2}}(L^\f{q}{q-2})}\\\leq& C \|\nabla\sqrt{\varrho} \|_{L^{\f{2pk}{2k(p-2)-p}}(L^{\f{2lq}{2l(q-2)-q}})}\|\varrho\|_{L^{k}(L^{l})}^{1/2}\\\leq& C \|\nabla\sqrt{\varrho} \|_{L^{\f{2pk}{2k(p-3)-p}}(L^{\f{2lq}{2l(q-3)-q}})}\|\varrho\|_{L^{k}(L^{l})}^{1/2}.
	\end{aligned}
\end{equation}
Plugging \eqref{3.91} and \eqref{3.92} into \eqref{3.81}, we get
 \begin{equation}\label{c3.11}\begin{aligned}
&\int_0^T\int \phi(t)v^{\varepsilon} \B[\partial_{t} (\varrho v )^{\varepsilon}-\partial_{t}(\varrho v^{\varepsilon})\B]\\\leq&
C\|v\|^{2}_{L^{p}(L^{q})}\\&
\left[\B(\|v\|_{L^{p}(L^{q})}+1\B)\|\nabla\sqrt{\varrho} \|_{L^{\f{2pk}{2k(p-3)-p}}(L^{\f{2lq}{2l(q-3)-q}})}\|\varrho\|_{L^{k}(L^{l})}^{1/2}+\| \nabla v\|_{L^{\f{pk}{k(p-2)-p}}(L^{\f{ql}{l(q-2)-q}})} \|\varrho\|_{L^{k}(L^{l})}\right],.
 \end{aligned}\end{equation}
 From Lemma \eqref{pLions}, we end up with, as $\varepsilon\rightarrow0$,
 $$
 \int_0^T\int \phi(t)v^{\varepsilon} \B[\partial_{t} (\varrho v )^{\varepsilon}-\partial_{t}(\varrho v^{\varepsilon})\B]\rightarrow0.$$
 In the light of the H\"older's inequality, we obtain
 \be\ba
 \|\varrho v\otimes v\|_{L^{\f{pk}{2k+p}}(L^{\f{ql}{2l+q}})}\leq \|v\|^{2}_{L^{p}(L^{q})}\|\varrho\|_{L^{k}(L^{l})}
 \ea\ee
  Using the integration by parts, we observe that
 \begin{equation}\label{3.11}\begin{aligned}
 		&\left|\int_0^T\int\phi(t)  \nabla v^{\varepsilon}  [(\varrho v\otimes v)^{\varepsilon}-(\varrho  v)\otimes v^{\varepsilon}]\right| \\
 		\leq &C\|\nabla v^\varepsilon\|_{L^{\f{pk}{k(p-2)-p}}(L^{\f{ql}{l(q-2)-q}})} \|(\varrho v\otimes v)^{\varepsilon}-(\varrho  v)\otimes v^{\varepsilon}\|_{L^{\f{pk}{2k+p}}(L^{\f{ql}{2l+q}})}\\
 		\leq & C\|\nabla v\|_{L^{\f{pk}{k(p-2)-p}}(L^{\f{ql}{l(q-2)-q}})} \left(\|(\varrho v \otimes v)^\varepsilon- \varrho v \otimes v\|_{L^{\f{pk}{2k+p}}(L^{\f{ql}{2l+q}})}+\|\varrho v\otimes v- \varrho v\otimes v^\varepsilon\|_{L^{\f{pk}{2k+p}}(L^{\f{ql}{2l+q}})}\right).
 \end{aligned}\end{equation}
 Hence, by the standard properties of the mollification, we have
 $$\int_0^T\int\phi(t)  \nabla v^{\varepsilon}  [(\varrho v\otimes v)^{\varepsilon}-(\varrho  v)\otimes v^{\varepsilon}] \rightarrow0 \text{ as }\varepsilon\rightarrow0.$$
 According to   H\"older's inequality on bound domain, we observe that
 $$\|\nabla(\varrho^\gamma)\|_{L^{\f{p}{p-1}}(L^{\f{q}{q-1}})}\leq
 C \|\varrho\|^{\f{2\gamma-1}{2}}_{L^{\f{pk(2\gamma-1)}{4k+p}}(L^{\f{ql(2\gamma-1)}{4l+q}})}\|\nabla\sqrt{\varrho} \|_{L^{\f{2pk}{2k(p-3)-p}}(L^{\f{2lq}{2l(q-3)-q}})},$$
 which in turn implies that
 \begin{equation}\label{d10}\begin{aligned}
 \int_0^T\int\phi(t) [v^{\varepsilon}\nabla(\varrho^\gamma)^{\varepsilon}-v\nabla (\varrho^\gamma)]\rightarrow0,
 \end{aligned}\end{equation}
 where we have used lemma \ref{lem2.11} and the condition $k\geq \frac{(\gamma-1)p}{2}$, $l\geq \frac{(\gamma-1)q}{2}$.

Then together with \eqref{c3.11}-\eqref{d10}, passing to the limits as $\varepsilon\rightarrow 0$, we know that
\begin{equation}\label{d11}
	\begin{aligned}
		-\int_0^T\int \phi_t \left(\frac{1}{2}\varrho |v|^2+\frac{\varrho^\gamma}{\gamma -1}\right)+\int_0^T\int\phi(t) \left(\mu |\nabla v|^2+(\mu+\lambda)|\Div v|^2\right)=0.	
	\end{aligned}
\end{equation}
The next objective is to get the energy equality up to the initial time $t=0$ by the similar method in \cite{[CLWX]} and \cite{[Yu2]}, for the convenience of the reader and the integrity of the paper, we give the details. First we prove the continuity of $\sqrt{\varrho}v(t)$ in the strong topology as $t\to 0^+$.  To do this, we define the function $f$ on $[0,T]$ as
$$f(t)=\int_{\mathbb{T}^d}(\varrho v)(t,x)\cdot \phi(x) dx,\ for\ any\ \phi(x)\in \mathfrak{D}(\mathbb{T}^d),$$
which is a continuous function with respect to $t\in [0,T]$. Moreover, since
$$\varrho \in L^\infty{0,T; L^\gamma (\mathbb{T}^n)}\ and \ \sqrt{\varrho}v\in L^\infty(0,T;L^2(\mathbb{T}^n)),$$
we can obtain $\varrho v\in L^\infty(0,T;L^{\frac{2\gamma }{\gamma+1}}(\mathbb{T}^d)).$\\
From the moument equation, we have
$$\frac{d}{dt}\int_{\mathbb{T}^d} (\varrho v)(t,x)\cdot \phi(x) dx=\int_{\mathbb{T}^d}\varrho v\otimes v:\nabla \phi(x)-P\Div\phi(x)-\mu\nabla v\nabla \phi(x)-(\mu+\lambda )\Div v\Div \phi(x)dx,$$
which is bounded for any function $\phi\in \mathfrak{D}(\mathbb{T}^d)$. Then it follows from the Corollary 2.1 in  \cite{[Feireisl2004]} that
\begin{equation}\label{d14}
	\varrho v\in C([0,T];L^{\frac{2\gamma}{\gamma +1}}_{\text{weak}}(\mathbb{T}^d)).
\end{equation}
On the other hand, we derive  from the mass equation $\eqref{INS}_1$ that
\begin{equation}\label{d12}
	\begin{aligned}
		\partial_t(\varrho^\gamma )=-\gamma \varrho ^\gamma \Div v-2\gamma \varrho^{\gamma -\frac{1}{2}}v\cdot \nabla \sqrt{\varrho},	
	\end{aligned}
\end{equation}
and
\begin{equation}\label{d13}
	\begin{aligned}
		\partial_t(\sqrt{\varrho})=\frac{\sqrt{\varrho }}{2}\Div v+v\cdot \nabla \sqrt{\varrho},
	\end{aligned}
\end{equation}
which together with \eqref{key0} gives
$$\partial_t\varrho^\gamma\in L^{\f{kp}{k(p-2)+(\gamma-1)p}}(L^{\f{lq}{l(q-2)+(\gamma-1)q}}),\ \ \ \nabla \varrho^\gamma\in L^{\f{kp}{k(p-3)+(\gamma-1)p}}( L^{\f{lq}{l(q-3)+(\gamma-1)q}}),$$
and
$$\partial_t\sqrt{\varrho}\in L^{\f{2kp}{2k(p-2)-p}}(L^{\f{2lq}{2l(q-2)-q}}),\ \ \ \nabla \sqrt{\varrho}\in L^{\f{2kp}{2k(p-3)-p}}( L^{\f{2lq}{2l(q-3)-q}}),$$
Hence, using the Aubin-Lions Lemma \ref{AL}, we can obtain
\begin{equation}\label{d15}
	\varrho^\gamma\in C([0,T];L^{\frac{lq}{l(q-3)+(\gamma-1)q}}((\mathbb{T}^d)) \,and\,\ \sqrt{\varrho }\in C([0,T];L^{\f{2lq}{2l(q-3)-q}}(\mathbb{T}^d)),\
\end{equation}
for $k\geq \frac{(\gamma-1)(d+q)p}{2q-d(p-3)}$, $p>3$ and $q>\max\{3, \frac{d(p-3)}{2}\}$.

Meanwhile, using the natural energy \eqref{energyineq}, \eqref{d14} and \eqref{d15}, we have
\begin{equation}\label{d16}
	\begin{aligned}
		0&\leq \overline{\lim_{t\rightarrow 0}}\int |\sqrt{\varrho} v-\sqrt{\varrho_0}v_0|^2 dx\\
		&=2\overline{\lim_{t\rightarrow 0}}\left(\int \left(\f{1}{2}\varrho |v|^2 +\f{1}{\gamma -1}\varrho ^\gamma \right)dx-\int\left(\f{1}{2}\varrho_0 |v_0|^2+\f{1}{\gamma -1}\varrho_0 ^\gamma \right)dx\right)\\
		&\ \ \ +2\overline{\lim_{t\rightarrow 0}}\left(\int\sqrt{\varrho_0}v_0\left(\sqrt{\varrho_0}v_0-\sqrt{\varrho} v\right)dx+\f{1}{\gamma -1}\int \left(\varrho_0^\gamma -\varrho^\gamma\right)dx\right)\\
		&\leq 2\overline{\lim_{t\rightarrow 0}}\int \sqrt{\varrho_0}v_0\left(\sqrt{\varrho_0}v_0-\sqrt{\varrho}v\right)dx\\
		&=2\overline{\lim_{t\rightarrow 0}}\int v_0 \left(\varrho_0 v_0 -\varrho v\right)dx+2\overline{\lim_{t\rightarrow 0}}\int v_0 \sqrt{\varrho }v\left(\sqrt{\varrho }-\sqrt{\varrho_0}\right)dx=0,
	\end{aligned}
\end{equation}
from
which it follows
\begin{equation}\label{d17}
	\sqrt{\varrho} v(t)\rightarrow \sqrt{\varrho }v(0)\ \ strongly\ in\ L^2(\Omega)\ as\ t\rightarrow 0^+.
\end{equation}
Similarly, one has the right temporal continuity  of $\sqrt{\varrho}v$ in $L^2(\Omega)$, hence, for any $t_0\geq 0$, we infer that
\begin{equation}\label{d18}
	\sqrt{\varrho} v(t)\rightarrow \sqrt{\varrho }v(t_0)\ \ strongly\ in\ L^2(\Omega)\ as\ t\rightarrow t_0^+.
\end{equation}
Before we go any further, it should be noted that \eqref{d11} remains valid for function $\phi$ belonging to $W^{1,\infty}$ rather than $C^1$, then for any $t_0>0$, we redefine the test function $\phi$ as $\phi_\tau$ for some positive $\tau$ and $\alpha $ such that $\tau +\alpha <t_0$, that is
\begin{equation}
	\phi_\tau(t)=\left\{\begin{array}{lll}
		0, & 0\leq t\leq \tau,\\
		\f{t-\tau}{\alpha}, & \tau\leq t\leq \tau+\alpha,\\
		1, &\tau+\alpha \leq t\leq t_0,\\
		\f{t_0-t}{\alpha }, & t_0\leq t\leq t_0 +\alpha ,\\
		0, & t_0+\alpha \leq t.
	\end{array}\right.
\end{equation}
Then substituting this test function into \eqref{d11}, we arrive at
\begin{equation}
	\begin{aligned}
		-\int_\tau^{\tau+\alpha}\int& \f{1}{\alpha}\left(\f{1}{2}\varrho v^2+\f{1}{\gamma-1}\varrho^\gamma \right)+\f{1}{\alpha}\int_{t_0}^{t_0+\alpha}\int 	\left(\f{1}{2}\varrho v^2+\f{1}{\gamma-1}\varrho^\gamma \right)\\
		&+\int_{\tau}^{t_0+\alpha}\int \phi_\tau \left(\mu |\nabla v|^2+\left(\mu+\lambda\right)|\Div v|^2\right)=0.
	\end{aligned}	
\end{equation}
Taking $\alpha\rightarrow 0$ and using  the fact that $\int_0^t\int\left(\mu |\nabla v|^2+\left(\mu+\lambda\right)|\Div v|^2\right)$ is continuous with respect to $t$ and the  Lebesgue point Theorem, we deduce that
\begin{equation}
	\begin{aligned}
		-\int&\left(\f{1}{2}\varrho v^2+\f{1}{\gamma-1}\varrho^\gamma \right)(\tau)dx+\int\left(\f{1}{2}\varrho v^2+\f{1}{\gamma-1}\varrho^\gamma \right)(t_0)dx\\
		&+\int_\tau^{t_0}\int\left(\mu |\nabla v|^2+\left(\mu+\lambda\right)|\Div v|^2\right)=0.
	\end{aligned}
\end{equation}
Finally, letting $\tau\rightarrow 0$, using the continuity of $\int_0^t\int\left(\mu |\nabla v|^2+\left(\mu+\lambda\right)|\Div v|^2\right)$, \eqref{d14} and \eqref{d17}, we can obtain
\begin{equation}\ba
	\int\left(\f{1}{2}\varrho v^2+\f{1}{\gamma-1}\varrho^\gamma \right)(t_0)dx+\int_0^{t_0}\int&\left(\mu |\nabla v|^2+\left(\mu+\lambda\right)|\Div v|^2\right)dxds\\=&\int\left(\f{1}{2}\varrho_0 v_0^2+\f{1}{\gamma-1}\varrho_0^\gamma \right)dx.
	\ea\end{equation}
Then we complete the proof of Theorem \ref{the1.5}.
\end{proof}

\section*{Acknowledgement}
The authors would like to express their sincere gratitude to Prof. Quansen Jiu for pointing
out this problem to us.
Ye was partially supported by the National Natural Science Foundation of China  under grant (No.11701145) and China Postdoctoral Science Foundation (No. 2020M672196).
 Wang was partially supported by  the National Natural
 Science Foundation of China under grant (No. 11971446, No. 12071113   and  No.  11601492).
Yu was partially supported by the
National Natural Science Foundation of China (NNSFC) (No. 11901040), Beijing Natural
Science Foundation (BNSF) (No. 1204030) and Beijing Municipal Education Commission
(KM202011232020).

\end{document}